\documentclass[12pt,reqno,a4paper]{amsproc}
\usepackage{amsthm,wrapfig,amsmath,amssymb,amsfonts,setspace,verbatim,graphicx,mathtools,mathrsfs,commath,float,mdframed,frame,xcolor,qtree,enumitem, ulem,xfrac,afterpage}
\usepackage[hidelinks]{hyperref}
\usepackage[T1]{fontenc}

\usepackage[margin=1.3in]{geometry}

\DeclareMathOperator{\gr}{Gr}
\DeclareMathOperator{\diam}{diam}

\newtheorem{ut}{Theorem}
\newtheorem{up}[ut]{Proposition}
\newtheorem{ul}[ut]{Lemma}

\newtheorem{uc}[ut]{Corollary}

\newtheorem{ue}{Example}

\newtheorem{uq}{Question}

\theoremstyle{remark}

\theoremstyle{definition}
\newtheorem{ur}{Remark}

\author[D.S. Lipham]{David S. Lipham}

\address{Department of Mathematics and Data Science, College of Coastal Georgia, Brunswick GA 31520, United States of America}
\email{dlipham@ccga.edu}



\title[Endpoints of smooth plane dendroids]{Endpoints of smooth plane dendroids\\ 
\textnormal{\small \textit{In memory of professors Gary Gruenhage and Piotr Minc}}}

\begin{document}


\begin{abstract}Let $X$ be a smooth  dendroid in the plane $\mathbb R^2$. We show that each endpoint of $X$ is arcwise accessible from $\mathbb R^2\setminus X$, and that the space of  endpoints $E(X)$ has the  property of a circle. In the event that  $E(X)$ is connected,  we call $X$ a  \textit{Bellamy dendroid}. We prove that if $E(X)$ is 1-dimensional, then $X$ contains a Bellamy dendroid or a Cantor set of arcs.   In particular, if $E(X)$ totally disconnected and $1$-dimensional, then $X$  is non-Suslinian. An example is constructed to show that this is false outside the plane. \end{abstract}

\maketitle

\

\section{Introduction}

Dendroids  form an important class of uniquely arcwise connected  continua.   
They can be defined  as hereditarily arcwise connected continua without simple closed curves.  Every contractible $1$-dimensional continuum is a dendroid \cite{cont}.  Dendroids are tree-like \cite{tree} and have the fixed point property \cite{mank,bor}.   For a nice survey with open problems, see \cite{vn}.

The focus of this paper is on smooth plane dendroids. These are dendroids which embed into the plane and  admit radially convex metrics \cite{on}.  We will show that each endpoint of a smooth plane dendroid $X$ is accessible  from $\mathbb R^2\setminus X$, with at most  one exception (\S3). From here we find that the space of endpoints $E(X)$ is circle-like, in that every two of its points are separated by two other points (\S4). Bellamy constructed in \cite{bel}  a smooth plane dendroid with connected endpoint set which he observed to have this property.

Next we examine the topological dimension of  endpoints. It is well known that the endpoints of a smooth plane dendroid can be $1$-dimensional even when they are totally disconnected, as seen in the Lelek fan \cite{lel}. Another property of the Lelek fan is that it contains an uncountable collection of pairwise-disjoint arcs, i.e.\ it is non-Suslinian. This observation leads to the following theorem: Each  smooth plane dendroid with  hereditarily disconnected and 1-dimensional endpoint set  must be non-Suslinian (\S5). Planarity is critical to this result: We construct in \S5.3 a (not planable) Suslinian smooth dendroid $D$ such that $E(D)$ is homeomorphic to the endpoints of the Lelek fan. 

Finally, in \S6 we examine conditions under which a smooth plane dendroid $X$ must contain a \textit{Bellamy dendroid} (i.e.\ one whose endpoint set is connected). Our main result is that if $E(X)$ is 1-dimensional, then $X$ must contain a Bellamy dendroid or a Cantor set of arcs. We do not know if every Bellamy dendroid contains a Cantor set of arcs; see Question 1 in \S7. 

\subsection*{Acknowledgements}The example presented in \S5.3 is attributed to Ed Tymchatyn and Piotr Minc, although it was not published until now. I am grateful to Ed Tymchatyn for sharing with me the ideas for its construction. 

I also wish to thank Logan Hoehn for suggesting the use of colocal connectedness to prove accessibility of endpoints.

\section{Fundamental notions}

A \textbf{continuum} is a compact connected metric space. A continuum  is \textbf{hereditarily unicoherent} if every two of its subcontinua have connected intersection.  

A \textbf{dendroid} is a hereditarily unicoherent, arcwise connected continuum. It is easy to see that dendroids are uniquely arcwise connected. Furthermore,  in the plane a continuum $X\subset \mathbb R^2$  is a dendroid if and only if  $X$ is uniquely arcwise connected and non-separating \cite[Theorem 1.5]{mank}.
Given  a dendroid $X$ and points $x,y\in X$ then we let $\alpha(x,y)$ denote the unique arc in $X$ with endpoints $x$ and $y$.

  A dendroid $X$ is \textbf{smooth} if there exists $p\in X$ such that if $x_n\to x$ in $X$, then $\alpha(x_n,p)\to\alpha(x,p)$ in the Hausdorff distance. The point $p$ is called an \textbf{initial point} of $X$; alternatively we say that $X$ is \textbf{smooth at $p$}. 
  
If $X$ is a dendroid smooth at $p$, then there exists a compatible metric $d$ on $X$ such that  $d(y,p)<d(x,p)$ whenever $y\in \alpha(x,p)$ and $y\neq x$    \cite[Theorem 10]{on}. 
 This type of  metric is called \textbf{radially convex} with respect to $p$.

A point $e\in X$ is an \textbf{endpoint}  if $e$ is an endpoint of every arc in $X$ that contains it.  We let $E(X)$ denote the set of all endpoints of $X$. By \cite[Lemma 3]{nik2}, each arc in $X$ is contained in a maximal arc in $X$, each of whose endpoints belong to $E(X)$. The endpoint set  of a smooth dendroid is always $G_{\delta}$ \cite{borel}. 
  
The \textbf{order} of a point $x\in X$ is defined to be the number of arc components of $X\setminus \{x\}$.  A point order at least $3$ is called a \textbf{ramification point}. The set of all ramification points is denoted $R(X)$.
 
A \textbf{fan} is a dendroid $X$ with only one ramification point $x$ (thus  each component of $X\setminus \{x\}$ is homeomorphic to the interval $(0,1]$).  The \textbf{Cantor fan} is the quotient of $C\times [0,1]$ that takes  $C\times \{0\}$ to a  point. The \textbf{Lelek fan} is a smooth fan with  dense set of endpoints \cite{lel}. 

A continuum is \textbf{Suslinian} if  it contains no uncountable collection of non-degenerate, pairwise-disjoint subcontinua \cite{lel1}. The two propositions below help to understand the Suslinian property in dendroids.

\begin{up}[{cf.\ \cite[Theorem 9.9]{nik4}}]Let $X$ be a  dendroid. The following are equivalent.
\begin{itemize}
\item[(1)] $X$ is Suslinian;
\item[(2)] $X\setminus E(X)$ is a countable union of arcs;
\item[(3)]$R(X)$ is countable and each $r\in R(X)$ has countable order.
\end{itemize} \end{up}

\begin{proof} 

 (1)$\Rightarrow$(2): If $X$ is Suslinian, then  there is a countable set $Q\subset X$ which intersects each arc of $X$ \cite[Corollary 2.3]{lel1}. Fix $x\in X\setminus E(X)$. If $y\in X\setminus E(X)$ then the maximal arc $\beta$ extending $\alpha(x,y)$ contains an arc $\alpha(y,e)$, where $e\in E(X)$. There exists $q\in Q\cap \alpha(y,e)\setminus \{e\}$. Then $y\in \alpha(x,q)\subset X\setminus E(X)$. This shows that $X\setminus E(X)$ is equal to the union of arcs $\alpha(x,q)$ with $q\in Q\setminus E(X)$. 
 
(2)$\Rightarrow$(1): Suppose that $X\setminus E(X)=\bigcup_{n=1}^\infty \alpha_n$ is a countable union of arcs. Let $Q_n$ be a countable dense subset of $\alpha_n$. Given any arc  $\beta\subset X\setminus E(X)$, by Baire's theorem, there exists $n$ such that $\alpha_n$ contains an arc of $\beta$. Then $Q_n \cap \beta\neq\varnothing$. This shows that the countable set $Q=\bigcup_{n=1}^\infty Q_n$ intersects each arc of $X$. Therefore $X$ is Suslinian.

  (1)$\Rightarrow$(3): Suppose that  $X$ is Suslinian.  Then clearly each ramification point must have countable order. Further, by \cite[Theorem 1]{nik3}  $R(X)$ is covered by countably many  arcs, and  along each  arc there can be only countably many ramification points. Therefore $R(X)$ is countable. 
  
 (3)$\Rightarrow$(1):  Let $\mathscr A$ be the set of all non-degenerate arc components of $X\setminus (R(X)\cup E(X))$. Each element of $\mathscr A$ is homeomorphic to the interval  $(0,1)$ and has at least one endpoint in $R(X)$, unless $X$ is an arc. Thus if $\mathscr A$ were uncountable, then uncountably many  components would end at the same ramification point $r$, which would contradict that the order of $r$ is countable.  So $\mathscr A$ is countable. Let $D$ consist of a countable dense subset from each element of $\mathscr A$. Then $Q=R(X)\cup D$ is a countable set that meets every arc in $X$. Therefore $X$ is Suslinian. \end{proof}
 
 By a \textbf{Cantor set of arcs} 
 we mean a continuous collection of pairwise-disjoint arcs whose decomposition space is a Cantor set.
 
 \begin{up}Let $X$ be a plane dendroid. The following are equivalent:
 \begin{itemize}
\item[(i)] $X$ is non-Suslinian;
\item[(ii)] $X$ contains a Cantor set of arcs.
 \end{itemize}\end{up}
 
 \begin{proof}(ii)$\Rightarrow$(i) is trivial. For (i)$\Rightarrow$(ii), suppose that $X$ is non-Suslinian. By \cite[Theorem 2.1] {wea}, there is a closed $A\subset X$ such that the components of $A$ are non-degenerate,  the
decomposition of $A$ into components is continuous, and the space of components of $A$ is a Cantor set $C$. By Moore's triod theorem, there is a countable  $Q\subset C$ such that the components associated with $C\setminus Q$ are arcs.  $C\setminus Q$ is an uncountable Borel set, thus it contains a Cantor set $D$. The components of $A$ over $D$ form a Cantor set of arcs. \end{proof}

\section{Accessibility of endpoints}

 If $U$ is an open subset of $\mathbb R^2$ then a point $z\in \mathbb R^2\setminus U$ is \textbf{accessible  from} $U$ if there is an arc $\alpha\subset \mathbb R^2$ such that $\alpha\setminus U=\{z\}$.  When a continuum $X$ is understood from context, we will say that $x\in X$ is \textbf{accessible} if $x$ is accessible from $\mathbb R^2\setminus X$.

A continuum $X$ is \textbf{colocally connected} at $x\in X$ if for every open set $V$ containing $x$ there is an open $U\subset V$ such that $x\in U$ and $X\setminus U$ is connected.

A \textbf{simply connected domain} is a bounded open subset of the plane which is connected and simply connected (and is  homeomorphic to $\mathbb R^2$).

\begin{ul}Let $X$ be a dendroid in the plane. Let  $U$ be an open set in the plane such that $X\setminus U$ is  connected. If  $W$ is any connected component of $U$, then $X\setminus W$ is connected.\end{ul}

\begin{proof}Let $x,y\in X\setminus W$. Let $f_x:[0,1]\hookrightarrow X$ such that $f_x(0)=x$ and $f_x(1)\in W$. Let $a=\inf\{t\in [0,1]:f(t)\notin U\}$. Define $f_y$ and $b$ similarly for $y$. Let $\alpha$ be an arc in $X\setminus U$ from $f_x(a)$ to $f_y(b)$. Then $\beta=f_x[0,a]\cup \alpha\cup f_y[0,b]$ is an arc in $X\setminus W$ from  $x$ to $y$. This shows that $X\setminus W$ is connected. \end{proof}

\begin{ul}Let $X$ be a dendroid in the plane. Let $S\subset \mathbb R^2$ be a circle with complementary components $U$ bounded and $V$ unbounded, such that $X$ meets $V$. Suppose that $W$ is a connected open subset of $U$ such that   $X\setminus W$ is connected.  Then there is a simply connected domain  $W'\subset U$ such that $W\subset W'$ and $X\setminus W'$ is connected\end{ul}

\begin{proof}Let $W'$ be the union of $W$ with all of its bounded complementary components; clearly $W\subset W'\subset U$. Each complementary component of $W$ meets $\partial W$ (e.g.\ by  \cite[Theorem 5.4]{nad}), so  $W'$ is connected. Also $W'$ is open and has connected complement (because $\mathbb R^2\setminus W'$ is just the unbounded component of $\mathbb R^2\setminus W$).  Therefore $W'$ is simply connected. Note that $X\setminus W$ lies wholly in the unbounded component of $\mathbb R^2\setminus W$ because  $X\setminus W$  is a connected set that meets $V$. Therefore $X\setminus W'=X\setminus W$ is connected.\end{proof}

\begin{ul} Let $X$ be a dendroid in the plane and  $W\subset \mathbb R^2$ a simply connected domain. If    $X\setminus W$ is connected, then $W\setminus X$ is connected. \end{ul}

\begin{proof}Suppose that $X\setminus W$ is connected but $W\setminus X$ is not.  Then the quotient $X/(X\setminus W)$ is  a dendroid that separates the $2$-sphere $\mathbb R^2/(\mathbb R^2\setminus W)$, a contradiction. \end{proof}

\begin{ut}\label{colo}Let $X$ be a plane dendroid. If $X$ is colocally connected at $x$, then $x$ is accessible.\end{ut}

\begin{proof}Let $\delta=\diam(X)$ and for each $n=1,2,3,\ldots$ let $S_n$ be the circle of radius $\delta/2n$ centered at $x$. Note that $X$ meets the unbounded component of each $\mathbb R^2\setminus S_n$.  Let $U_n$ be an open set in the bounded component of $\mathbb R^2\setminus S_n$ such that  $x\in U_n$ and $X\setminus U_n$ is connected. Let $W_n$ be the connected component of $x$ in $U_n$. By Lemma 3, $X\setminus W_n$ is connected. By Lemma 4  there is a simply connected domain $W'_n\subset U_n$ such that  $x\in W_n'$ and $X\setminus W'_n$ is connected. By Lemma 5, $W_n'\setminus X$ is path-connected. 

Now let $y_0\in \mathbb R^2\setminus X$. Let $\alpha_0\subset \mathbb R^2\setminus X$ be an arc from $y_0$ to  $y_1\in W'_1\setminus X$. Let $y_2\in W'_2\cap W'_1\setminus X.$  Let $\alpha_1\subset W'_1\setminus X$ be an arc from $y_1$ to $y_2$. Continue this process, letting $\alpha_n$ be an arc in $W'_n\setminus X$ from $y_{n}$ to a new point $y_{n+1}\in W'_n\cap W'_{n+1}\setminus X.$ The arcs $\alpha_n$ form a null sequence that converges to $x$. Hence the closure of their union is a locally connected continuum. It contains an arc $\alpha$ from $y_0$ to $x$, such that $\alpha\cap  X=\{x\}$. \end{proof}

\begin{uc}\label{acc}Let $X$ be a smooth plane dendroid with initial point $p$. Then every endpoint $e\in E(X)\setminus \{p\}$ is accessible.\end{uc}

\begin{proof}By \cite[Theorems 3.1 and 3.5]{km}, $X$ is colocally connected at $e$. By Theorem \ref{colo},   $e$ is accessible.\end{proof}

\begin{ur}The point $p$ could be an  inaccessible endpoint. See Figure 1.\end{ur}

\begin{ur}In Bellamy's dendroid \cite{bel}, the endpoints are the only accessible points. 
\end{ur}

\begin{figure}\includegraphics[scale=0.54]{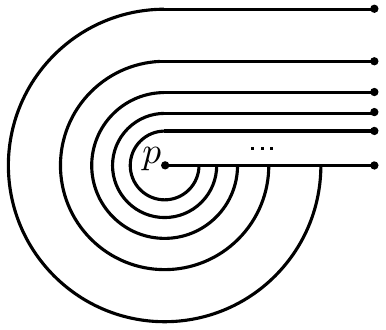}
\caption{A smooth plane dendroid in which $p$ is an inaccessible endpoint.}
\end{figure}

\section{Separation of endpoints}

\begin{figure}\includegraphics[scale=0.54]{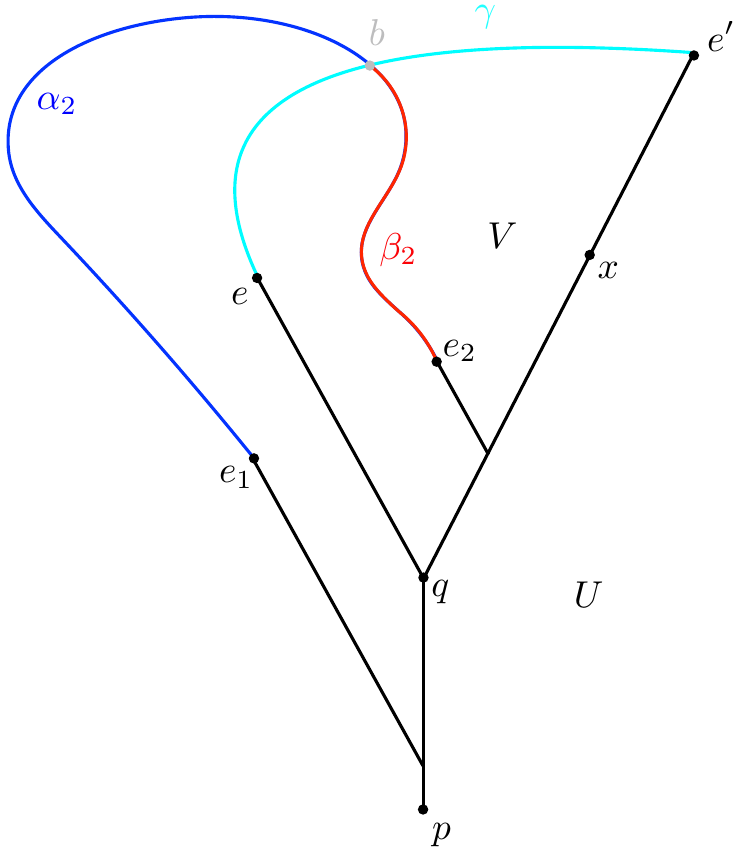}
\caption{Proof of Theorem \ref{circ}.}
\end{figure}

We   show that the endpoints of a smooth plane dendroid  have the circle-like property described in \S1. Then we examine separation properties of  hereditarily disconnected endpoint sets.  

 We begin with an easy consequence of the  $\mathrm{\theta}$-curve theorem \cite[Lemma 64.1]{mun}.
\begin{up}Let $O$ be a simple closed curve in the plane. Let $a,b,c,d\in O$ in cyclic order. Let $U$ and $V$ be the components of $\mathbb R^2\setminus O$. Let $\alpha$ and $\beta$ be arcs in $\overline U\setminus \{b,d\}$ and $\overline V\setminus \{b,d\}$, respectively, from $a$ to $c$. Then $\alpha\cup \beta$ contains a simple closed curve $\sigma$ which separates $b$ and $d$.\end{up}

\begin{ut}\label{circ}Let $X$ be a plane dendroid, smooth at $p$. Let $e\in E(X)\setminus \{p\}$. For every $x\in X\setminus \alpha(e,p)$  there is a simple closed curve $\sigma\subset \mathbb R^2$ which separates $e$ from $x$ and has the property $|\sigma\cap E(X)|\leq 2$ and $p\notin \sigma\cap E(X)$.\end{ut}

\begin{proof}Let $e'\neq p$ be an endpoint of a maximal arc in $X$ extending $\alpha(x,p)$. Let $\gamma$ be an arc from $e$ to $e'$ such that $\gamma\cap X=\{e,e'\}$, as provided by Corollary \ref{acc} and the fact that $\mathbb R^2\setminus X$ is path-connected.  There exists $q\in X$ such that $\alpha(q,p)=\alpha(e,p)\cap \alpha(x,p)$. The simple closed curve $\gamma\cup \alpha(e,q)\cup \alpha(e',q)$ separates the plane into two components $U$ and $V$. Without loss of generality,  $p\in \overline U$. Let $b\in \gamma$.

\smallskip

\textit{Claim 1: There is an arc $\alpha\subset \overline U \setminus \{e,x\}$ from $q$ to $b$  such that $\alpha\cap E(X)\subset \{e_1\}$ for some $e_1\neq p$.} 

\smallskip

Let us first assume that there exists a sequence of points $x_n\in U\cap X\setminus \alpha(p,q)$ such that $x_n\to p$. Since $X$ is smooth at $p$ we know  $\alpha(x_n,p)\to p$ and so eventually $\alpha(x_n,p)$ misses $e$ and $x$.  Fix $n$ sufficiently large and let $\alpha_1$ be a maximal arc in $X$ containing $\alpha(x_n,p)$. Then $\alpha_1$ has an endpoint $e_1\in U\setminus \{p\}$. By Corollary 6,  there is an arc $\alpha_2$ from $e_1$ to $b$ such that $\alpha_2\cap X=\{e_1\}$. Put $\alpha=\alpha(e_1,q)\cup \alpha_2$. 

In the other case that the sequence $x_n$ does not exist,  there is an open set $W\subset \mathbb R^2$ such that $p\in W$ and $W\cap U\cap X\subset \alpha(p,q)$. Now there exists $y\in \alpha(p,x)\setminus \{p,x\}$ which is accessible from $U\setminus X$. Then $\alpha(y,q)$ extends to an arc $\alpha$ from $q$ to $b$, which misses $e$ and $x$ and contains  no endpoints. 

\smallskip

\textit{Claim 2: There is an arc $\beta\subset \overline V \setminus \{e,x\}$ from $q$ to $b$ such that $\beta\cap E(X)\subset \{e_2\}$ for some $e_2\neq p$.}

\smallskip

 The proof is similar to that of Claim 1, using the fact that $X\cap \overline V$ is a dendroid smooth at $q$.  Assume that there exist $x_n\in V\cap X$ such that $x_n\to q$. Otherwise $q$ is accessible from $V$ and the claim becomes trivial. Eventually $\alpha(x_n,q)$ misses $e$ and $x$. The arc $\alpha(x_n,q)$ extends to an arc $\beta_1$ with an endpoint $e_2\in V\setminus \{p\}$. By Corollary 6, there is an arc $\beta_2$  from $b$ to $e_2$ such that $\beta_2\cap X=\{e_2\}$. Put $\beta=\alpha(e_2,q)\cup \beta_2$. 
 
 \smallskip

Now apply Proposition \ref{circ}  to $\alpha\cup \beta$.  See Figure 2.   \end{proof}

\begin{uc}Let $X$ be a  dendroid in the plane which is smooth at $p$. Within the space $E(X)\setminus \{p\}$, every two points are separated by two other points.\end{uc}

A space $X$ is \textbf{hereditarily disconnected} if $X$ contains no non-degenerate connected set, and  \textbf{totally disconnected} if every two points of $X$ are contained in disjoint clopen sets. The  \textbf{quasicomponent} of a point $x$ in a space $X$ is the intersection of all clopen subsets of $X$ which contain $x$.

 Clearly every totally disconnected space is hereditarily disconnected. We will prove a near converse for endpoint sets (see Theorem 12).

\begin{ul}\label{po}Let $Q$ be a quasicomponent of a space $X$. Let $p\in X$. Suppose that $U$ is open in $X$,  $\partial U\subset Q$, $|\partial U|=2$, and $U\setminus \{p\}$ is a non-empty subset of $Q$. Then $Q$ contains a non-degenerate connected subset of $X$.  \end{ul}

\begin{proof}Either $\overline U\subset Q$ or  $p\notin Q$.

\textit{Case 1:  $\overline U\subset Q$.} Let $\partial U=\{x,x'\}$ and let $x''\in U$. If $\overline U$ is not connected then it is a union of two disjoint closed sets $A$ and $B$ such that $x\in A$ and $x'\in B$. Without loss of generality, $x''\in B$.  If $B$ is not connected then it is the union of two non-empty disjoint closed sets $C$ and $D$ with $x'\in C$. Then $D$ is a clopen subset of $X$ containing some but not all of $Q$, a contradiction. So either $\overline U$ or $B$ must have been connected.

\textit{Case 2: $p\notin Q$.}  Let $A$ be a clopen subset of $X$ containing $Q$ and missing $p$. Then $V=A\cap U$ satisfies the hypotheses of the theorem, and $\overline V\subset Q$. By Case 1, $\overline V$ contains a non-degenerate connected set.\end{proof}

\begin{ut}\label{hdtd}Let $X$ be a smooth plane dendroid  with initial point $p$.  If $E(X)$ is hereditarily disconnected, then no two points of $E(X)\setminus \{p\}$ belong to the same quasicomponent of $E(X)$. In particular,  $E(X)\setminus \{p\}$ is totally disconnected. \end{ut}

 \begin{proof}Suppose that $Q$ is  a  quasicomponent   of $E(X)$ and  $|Q\setminus \{p\}|\geq 2$. We will show that $Q$ contains a non-degenerate connected set. To that end, let $e,e'\in Q\setminus \{p\}$. Let $\gamma$ be an arc with endpoints  $e$ and $e'$ such that $\gamma\cap X=\{e,e'\}$ (which exists by Theorem 6). Let $\sigma=\alpha(e,e')\cup \gamma$. Let $U$ and $V$ be the components of $\mathbb R^2\setminus\sigma$.

We claim that $U\cap E(X)\setminus \{p\}\subset Q$ or $V\cap E(X)\setminus \{p\}\subset Q$. If not, then  $U$ and $V$ each contain points of  $E(X)\setminus \{p\}$ outside of $Q$ (say $e_1$ and $e_2$), and by Theorem \ref{circ} there would exist a simple closed curve $\xi$ separating $e$ and $e'$, such that $\xi\cap E(X)=\{e_1,e_2\}$. Let $W$ be a component of $\mathbb R^2\setminus \xi$. There is a clopen  $A\subset E(X)$  such that $Q\subset A$  and $A$ misses both $e_1$ and $e_2$. Then $W\cap A$ is a clopen subset of $E(X)$ containing some but not all of $Q$, a contradiction.

Now assume that  $U\cap E(X)\setminus \{p\}\subset Q$. If this set is empty, then $V\cap E(X)\setminus \{p\}$ must be non-empty and contained in $Q$ by Theorem \ref{circ} and the arguments above. In either case,   $Q$ contains a non-degenerate connected set   by Lemma \ref{po}.
\end{proof}

\begin{ur}The following example shows that Corollary 10 and Theorem 12 cannot be improved by replacing $E(X)\setminus \{p\}$ with $E(X)$.   Consider the dendroid in Figure 1. Denote by $\alpha$ the horizontal arc containing $p$, whose opposite endpoint we label  $e$ (so $\alpha=\alpha(p,e)$).  Let $x_n$ denote the sequence of ramification points along $\alpha$ that converges to $p$. Let $e_n$ denote the sequence of endpoints that converges to $e$ at the right side of the figure. Replace each arc $\alpha(x_n,e_n)$ with a Lelek fan  $L_n$ whose vertex is attached at $x_n$, and which stretches all the way to $e_n$. This can be done so  that the $L_n$'s are disjoint and converge to $\alpha$. The endpoints of the resulting dendroid $X$ are $p$ and $e$, and the endpoints of  the individual $L_n$'s. Using the fact that $E(L_n)\cup \{x_n\}$ is connected (proved  in \cite{lel}), one can see that the quasicomponent of $p$ in $E(X)$ is $\{p,e\}$. Moreover, $p$ and $e$ cannot be separated by two other endpoints.  \end{ur}

\section{Dimension of endpoints}

\subsection{Preliminaries}

 A topological space $X$ is \textbf{zero-dimensional} if $X$ has a basis of clopen sets, and \textbf{almost zero-dimensional} if $X$ has a basis of neighborhoods which are intersections of clopen sets \cite{dvm}. Observe that almost zero-dimensional Hausdorff spaces are totally disconnected.

Zero-dimensional separable metric spaces  embed into $\mathbb R$ and are therefore linearly orderable. Endpoints of dendrites (locally connected dendroids) are known to be zero-dimensional.  Endpoints of smooth fans and $\mathbb R$-trees are known to be almost zero-dimensional \cite{ov2}.  The endpoint set of the Lelek fan  is  universal for almost zero-dimensional separable metric spaces \cite{dvm}.

A function $\varphi:Z\to [0,\infty)$ is \textbf{upper semi-continuous (USC)} if $\varphi^{-1}[0,t)$ is open for every $t>0$. The following is easily proved (cf.\ {\cite[Remark 4.2]{dvm}}).

\begin{up}\label{usc}Let $X$ be a topological space.  If $X$ is homeomorphic to the graph of a USC function with zero-dimensional domain, then $X$ is almost zero-dimensional. \end{up} 

\begin{proof}[Sketch of proof]If $Z$ is zero-dimensional and $\varphi:Z\to [0,\infty)$ is USC, then sets of the type $A\times [t,\infty)$, where $A$ is clopen in $Z$, yield a neighborhood basis of C-sets for the graph of $\varphi$.\end{proof}

A space $X$ is \textbf{zero-dimensional at} $x\in X$ if the point $x$ has a neighborhood basis of clopen sets. Put $$\Omega(X)=\{x\in X:X \text{ is zero-dimensional at }x\}$$ and $\Lambda(X)=X\setminus \Omega(X)$.

\begin{up}\label{sus}Let $X\subset \mathbb R^2$ be a Suslinian continuum in the plane. If $Y\subset X$ is almost zero-dimensional, then $Y$ is zero-dimensional. \end{up}

\begin{proof}Let $Y\subset X$ be almost zero-dimensional. There is a Polish universal almost zero-dimensional space \cite{dvm}, so by Lavrentiev's theorem we may assume that $Y$ is a $G_{\delta}$-set.  By \cite[Theorem A]{lvmmtv}, $\Lambda(Y)$ is countable. By \cite[Theorem 1]{co2h}, $Y$ is zero-dimensional.\end{proof}

  \subsection{Results} We will show that the conclusion of Theorem 12 can be strengthened from \textit{totally disconnected} to \textit{almost zero-dimensional}, and apply the result to Suslinian 
 dendroids.

\begin{ul}\label{fe}Let $X$ be a smooth dendroid in the plane and $ e\in E(X)\setminus \{p\}$. Let $\mathcal U_e$ be the collection of all  open subsets $U$ of $X$ such that $e\in U$ and $\partial U\cap E(X)=\varnothing$.  Put $$F(e) =\bigcap_{U\in \;\mathcal U_e} \overline{U}.$$
If $E(X)$ is hereditarily disconnected, then $F(e)\subset \alpha(e,p)$.\end{ul}

\begin{proof}Suppose $x\in X\setminus \alpha(e,p)$. By Theorem \ref{circ} there is an open  set $U$ such that  $e\in U$, $x\notin \overline {U}$,   $|\partial U\cap E(X)|\leq 2$ and $p\notin \partial U\cap E(X)$. Since $E(X)$ is hereditarily disconnected, by Theorem \ref{hdtd} we can find an open set $V$ such that  $e\in V$, $V\cap \partial U\cap E(X)=\varnothing$ and $ \partial V\cap E(X)=\varnothing$. Then $U\cap V\in \mathcal U_e$ and $x\notin \overline{U\cap V}$. Therefore $x\notin F(e)$. \end{proof}

\begin{ut}\label{az}Let $X$ be a smooth plane dendroid with initial point $p$. If $E(X)$ is hereditarily disconnected, then $E(X)\setminus\{p\}$ is almost zero-dimensional.  \end{ut} 

\begin{proof}Suppose that $E(X)$ is hereditarily disconnected, and let  $Y=E(X)\setminus\{p\}$. The collection of all clopen subsets of $Y$ is a basis for a  zero-dimensional topology $\mathcal W$ on $Y$.  Put  $Z=(Y,\mathcal W)$. Let $d$ be a radially convex metric on $X$ (with respect to $p$), and define $\varphi:Z\to [0,\infty)$ by $\varphi(e)=d(e,p).$

\smallskip

\textit{Claim 1: $\varphi$ is USC.} 

\smallskip

Let $t>0$. We show that $\varphi^{-1}[t,\infty)$ is closed in $Z$. To that end, let $e$ be any point in the $Z$-closure of $\varphi^{-1}[t,\infty)$. By separability of $X$ there is a countable sequence   $U_1(e)\supset U_2(e)\supset \ldots$  of sets in $ \mathcal U_e$ such that $$F(e)=\bigcap_{n=1}^\infty \overline{U_n(e)}.$$ For each $n$ there exists $e_n\in U_n(e)$ such that $\varphi(e_n)\geq t$. Let $x$ be any accumulation point of $(e_n)$ in $X$. Then $x\in F(e)$. By Lemma \ref{fe}, $x\in \alpha(e,p)$. By continuity of $d$ we have $d(x,p)\geq t$. Since $d$ is radially convex, this implies $\varphi(e)=d(e,p)\geq t$.  Therefore  $e\in \varphi^{-1}[t,\infty)$. 

\smallskip

\textit{Claim 2: $Y$ is homeomorphic to the graph of $\varphi$.}  

\smallskip
Consider the graph $\gr(\varphi)=\{\langle e,\varphi(e)\rangle:e\in Z\}$ as a subspace of $Z\times [0,\infty)$.  Apparently, $e\mapsto \langle e,\varphi(e)\rangle$ defines a continuous one-to-one mapping of $Y$ onto $\gr(\varphi)$. We will prove that its inverse is continuous by letting $A$ be a closed subset of $Y$ and showing that $\gr(\varphi\restriction A)=\{\langle e,\varphi(e)\rangle:e\in A\}$ is closed in $\gr(\varphi)$. To that end, suppose that $\langle e,\varphi(e)\rangle\in \gr(\varphi)$ is an accumulation point of  $\gr(\varphi\restriction A)$.   Let $U_1(e)\supset U_2(e)\supset \ldots$ be as above. For each $n$ there exists $\langle e_n,\varphi(e_n)\rangle\in (U_n(e)\cap A)\times (\varphi(e)-\sfrac{1}{n},\varphi(e)+\sfrac{1}{n}).$   Let $x$ be any  accumulation point of $(e_n)$ in $X$. Then $x\in \alpha(e,p)$ by Lemma \ref{fe}, and $\varphi(x)=\varphi(e)$. By radial convexity of $d$ we have $x=e$. So $e$ is an accumulation point of $A$ in the topology of $Y$. Since $A$ is closed,   $e\in A$. Therefore $\langle e,\varphi(e)\rangle\in \gr(\varphi\restriction A)$ as desired. 

\smallskip

By the preceding  claims and Proposition \ref{usc}, $Y$ is almost zero-dimensional. This completes the proof of Theorem 16. \end{proof}

\begin{ut}\label{main}Let $X$ be a Suslinian smooth plane dendroid. If  $E(X)$ is hereditarily disconnected, then $E(X)$ is zero-dimensional. \end{ut}

\begin{proof}Suppose that $E(X)$ is hereditarily disconnected. Let $p$ be an initial point of $X$. By Theorem \ref{az}, $E(X)\setminus \{p\}$ is almost zero-dimensional. By Proposition \ref{sus}, $E(X)\setminus \{p\}$ is zero-dimensional. By \cite[Corollary 1.3.5]{engd},  $E(X)$ is zero-dimensional.\end{proof}

\begin{uc}\label{hjkk}Let $X$ be a smooth plane dendroid. If  $E(X)$ is hereditarily disconnected and $1$-dimensional, then $X$ contains a Cantor set of arcs.\end{uc}

\subsection{Main example}

We construct a Suslinian smooth dendroid $D$ as a quotient of the Lelek fan $L$, such that $E(D)$ is homeomorphic to $E(L)$.  The example shows that Theorem 17 and Corollary 18 are false outside the plane. 

Let's begin by understanding what a Suslinian quotient of the Cantor fan might involve.  Let $C$ be the middle-thirds Cantor set. For each $n=0,1,2,\ldots$ let $\mathscr C^n$ be the natural partition of $C$ into $2^n$ disjoint closed sets of diameter $3^{-n}$. In $C\times [0,1]$, put $\langle c,0\rangle\sim \langle d,0\rangle$ for all $c,d\in C$. For each $n\geq 1$ define $\langle c,t\rangle\sim \langle d,t\rangle$ if $c$ and $d$ are in the same member of $\mathscr C^n$ and $$t\in [1-2^{1-n},1-2^{-n}].$$ The equivalence classes under the relation $\sim$ form an upper semi-continuous decomposition of the Cantor fan, and the quotient space is the Gehman dendrite (see Figure 3). Notice that the endpoints  are untouched by the identification of arcs; the endpoint set of the Gehman dendrite is  the  Cantor set.

\begin{figure}
\begin{center}\includegraphics[scale=0.28]{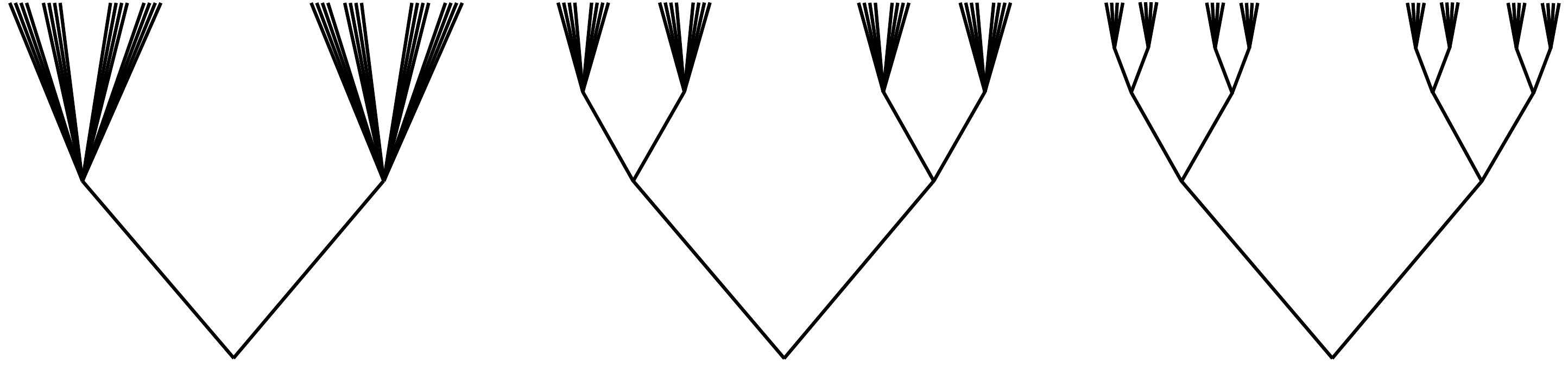}\end{center}
\caption{A quotient of the Cantor fan.}
 \end{figure}

In  order to apply a similar construction to the Lelek fan, we will require a few more definitions. If $\mathscr C$ and $\mathscr  D$ are partitions of a set $X$, then $\mathscr D$ \textbf{refines} $\mathscr C$ if for every $D\in \mathscr D$ there exists  $C\in \mathscr C$ such that $D\subset C$. If $\mathscr A$ is a collection of pairwise disjoint closed sets, then   $\mathscr C$ \textbf{respects} $\mathscr A$ if for every $C\in \mathscr C$ and $A\in \mathscr A$, either $C\subset A$ or $C\cap A=\varnothing$. A \textbf{null partition} of $X$ is a countable partition of $X$ whose elements form a null sequence (i.e.\ their diameters converge to $0$).

 \begin{ul}Let $\mathscr C=\{C_n:n<\omega\}$ be a closed partition of the Cantor set $C$, and suppose that $A_n$ is a closed subset of $C_n$ for each $n<\omega$. Then for any $\varepsilon>0$ there exists a null partition of $C$ of mesh $<\varepsilon$,  which refines  $\mathscr C$ and respects $\mathscr A=\{A_n:n<\omega\}$.\end{ul}

\begin{proof}Each $A_n$ is an intersection of clopen sets in $C$, so $C_n\setminus A_n$ can be written as a countable union of pairwise disjoint closed sets $B^n_0,B^n_1,\ldots$ and so forth.  We now have that $\{A_n:n<\omega\}\cup \{B^n_i:n,i<\omega\}$ is a countable partition of $C$ which can be enumerated $D_0,D_1,\ldots$ and so on. Each $D_k$ is partitioned by a finite collection $\mathscr D_k$ of pairwise disjoint closed sets of diameter $<\varepsilon/k$. The desired partition of $C$ is formed by the members of all $\mathscr D_k$'s. \end{proof}

 \begin{ue}There exists a Suslinian smooth dendroid $D$ such that $E(D)$ is homeomorphic to $E(L)$.\end{ue}

\begin{proof}  Given a function $\varphi:C\to [0,1]$ with Cantor set domain, we define $$L^\varphi_0=\bigcup_{c\in C} \{c\}\times [0,\varphi(c)].$$ We consider the Lelek fan $L$ as the quotient of some $L^\varphi_0$ that is obtained by shrinking $C\times \{0\}$ to a single point (the vertex of the fan). Thus $E(L)=\{\langle c,\varphi(c)\rangle:\varphi(c)>0\}$.

 In $L^\varphi_0$, put $\langle c,0\rangle\sim \langle d,0\rangle$ for all $c,d\in C$.  Let $A_1=\{c\in C:\varphi(c)\geq \frac{3}{4}\}$. Let $\mathscr C^1$ be a null partition of $C$ that respects $A_1$.  For each $c\in C$ let
$$ \varphi_1(c)=\begin{cases}
			\frac{1}{2} & \text{if $c\in A_1$}\\
            0 & \text{otherwise.}
		 \end{cases}$$
		   If $c$ and $d$ belong to the same member of  $\mathscr C^1$, then put $\langle c,t\rangle\sim \langle d,t\rangle$ for each $t\in [0,\varphi_1(c)]$. We continue this procedure as  follows. 
		   
		   Suppose $n\geq 2$, and $\mathscr C^{n-1}=\{C^{n-1}_k:k<\omega\}$ is a partition of  $C$. Let $$A^{n}_k=\big\{c\in C^{n-1}_k:\varphi(c)\geq \varphi_{n-1}(c)+\textstyle \frac{3}{4n}\big\}.$$  Put $$
\varphi_n(c)=\begin{cases}
			\varphi_{n-1}(c)+\frac{1}{2n} & \text{if $c\in \bigcup_{k=0}^\infty A^{n}_k$}\\
           \varphi_{n-1}(c) & \text{otherwise.}
		 \end{cases}
$$
Let $\mathscr C^n$ be a null partition of mesh $<\frac{1}{n}$ that refines $\mathscr C^{n-1}$ and respects $\mathscr A^n$.  If $c$ and $d$ belong to the same member  of $\mathscr C^n$, put $\langle c,t\rangle\sim \langle d,t\rangle$ for each $t\in [\varphi_{n-1}(c),\varphi_{n}(c)]$.

It is easily checked that the equivalence classes under $\sim$ form an upper semi-continuous decomposition of $L$ (note that a sequence of equivalence classes must converge to an entire  equivalence class, or to a single point of $L$). Moreover  $D$ is a Suslinian smooth dendroid with endpoint set    $E(L)$. \end{proof}

\begin{ur}The dendroid $D$ not planable by Theorem \ref{main}. In fact, since $E(L)$ is almost zero-dimensional and $1$-dimensional, by Proposition \ref{sus} every plane continuum that homeomorphically contains $E(L)$ is non-Suslinian.
\end{ur}

\begin{ur}A space is called \textbf{rational} if it has a basis of open sets with countable boundaries. It is known that every rational continuum is Suslinian, and that the converse is false.   The  dendroid $D$ is an example of a Suslinian but not rational continuum, owing to the fact that $E(L)$ is not rational  \cite[Corollary 4.8]{coh}. Lelek  constructed a Suslinian smooth \textit{plane}  dendroid  which is not rational  \cite[Example 3.1]{lel1}.  \end{ur}

\section{Bellamy dendroids}

Motivated by \cite{bel}, we  define a \textbf{Bellamy dendroid} to be a smooth plane dendroid with connected endpoint set. The following theorem allows us to restate the results of \S5.2  in terms of  Bellamy dendroids. 

\begin{ut}\label{hdbel}Let $X$ be a smooth plane dendroid. Then $E(X)$ is hereditarily disconnected if and only if $X$ does not contain a Bellamy dendroid.\end{ut}

\begin{proof}Suppose that $E(X)$ is not hereditarily disconnected, and let $C$ be a non-degenerate connected subset of $E(X)$. Then $K=\overline{C}$ is a smooth plane dendroid \cite[Corollary 6]{on}, and  $C$ is a dense connected subset of $E(K)$. Therefore $E(K)$ is connected and $K$ is a Bellamy dendroid. 

The other direction is trivial.\end{proof}

\begin{ut}If $X$ is a  smooth plane dendroid, then:
\begin{itemize} 
\item[(a)] $E(X)\setminus \{p\}$ is almost zero-dimensional, or
\item[(b)] $X$ contains a Bellamy dendroid.
\end{itemize}\end{ut}

\begin{proof}If $X$ does not contain a Bellamy dendroid, then $E(X)\setminus \{p\}$ is almost zero-dimensional by   Theorems \ref{az} and \ref{hdbel}.  \end{proof}

\begin{ut}\label{hjk}If $X$ is a Suslinian smooth plane dendroid, then:
\begin{itemize} 
\item[(a)] $E(X)$ is zero-dimensional, or
\item[(b)] $X$ contains a Bellamy dendroid.
\end{itemize}\end{ut}

\begin{proof}Theorems \ref{main} and \ref{hdbel}.   \end{proof}

\begin{uc}Let $X$ be a smooth plane dendroid. If $E(X)$ is 1-dimensional, then $X$ contains a Bellamy dendroid or  a Cantor set of arcs.
\end{uc}

\begin{proof}Corollary \ref{hjkk} and Theorem \ref{hdbel}.\end{proof}

\section{Questions}We do not know if all Bellamy dendroids are non-Suslinian.

\begin{uq}Is there a Suslinian Bellamy dendroid? \end{uq}

\begin{uq}Let $X$ be a smooth plane dendroid. If $X$ is Suslinian, then is some point of $X\setminus E(X)$ accessible from $\mathbb R^2\setminus X$? \end{uq}

A positive answer to Question 2 may lead to a negative answer to Question 1.  Recall  that  Bellamy's dendroid in \cite{bel} has no accessible points other than endpoints.

\end{document}